\documentclass[10pt]{amsart}
\UseRawInputEncoding
\usepackage{amsmath}
\usepackage{amsthm}
\usepackage{amscd}
\usepackage{amsfonts}
\usepackage{amssymb}
\usepackage{graphicx}
\usepackage{color}
\usepackage[bookmarksnumbered, colorlinks, plainpages]{hyperref}
\usepackage{hyperref}
\hypersetup{colorlinks=true,linkcolor=red, anchorcolor=green, citecolor=cyan, urlcolor=red, filecolor=magenta, pdftoolbar=true}

\newtheorem{theorem}{Theorem}[section]
\newtheorem{lemma}[theorem]{Lemma}
\newtheorem{proposition}[theorem]{Proposition}
\newtheorem{corollary}[theorem]{Corollary}
\theoremstyle{definition}
\newtheorem{definition}[theorem]{Definition}
\newtheorem{example}[theorem]{Example}

\newtheorem{remark}[theorem]{Remark}
\numberwithin{equation}{section}

\begin{document}
\setcounter{page}{1}

\title[Spectral properties of $(m,n)$-Isosymmetric  multivariable operators]{Spectral properties of $(m,n)$-Isosymmetric  multivariable operators}

\author[ Sid Ahmed O. A. Mahmoud, A. Bachir, S. Mecheri and A. Segres]{Sid Ahmed Ould Ahmed Mahmoud, Ahmed Bachir, Salah Mecheri and Abdelkader Segres}

\address{ Sid Ahmed Ould Ahmed Mahmoud  \endgraf
  Mathematics Department, College of Science, Jouf
University,\endgraf Sakaka P.O.Box 2014. Saudi Arabia}
\email{sidahmed@ju.edu.sa}
\address{ Ahmed Bachir \endgraf Department of Mathematics, College of Science, King Khalid University, \endgraf Abha, Saudi
		Arabia.}
	\email{abishr@kku.edu.sa, bachir1960@icloud.com}

\address{Salah Mecheri\endgraf
Department of Mathematics\endgraf
Mohamed El Bachir Elibrahimi\endgraf University
Bordj Bou Arreridj  Algeria}
\email{e-mail:mecherisalah@hotmail.com}

	\address{Abdelkader Segres \endgraf Department of Mathematics, University of Mascara, Mascara, Algeria}
	
	\email{segres03@yahoo.fr}

\keywords{: Hilbert space, $m$-isometry, $n$-ymmetry, $(m,n)$-isosymmetric operators}
\subjclass[2010]{ 47A05, 47A10, 47A11}

\begin{abstract}
 Inspired by recent works on $m$-isometric and $n$-symmetric multivariables  operators on Hilbert spaces, in this paper we introduce the class of $(m, n)$-isosymmetric multivariables operators. This new class of operators emerges as a generalization of the $m$-isometric and $n$-isosymmetric multioperators.
 We study this class of operators and give some of their basic properties. In particular, we show that if ${\bf \large R} \in {\mathcal B}^{(d)}({\mathcal H})$ is an $(m,n
)$-isosymmetric multioperators and ${\bf \large Q}\in {\mathcal B}^{(d)}({\mathcal H})$  is an $q$-nilpotent multioperators,
 then ${\bf\large R} +{\bf\large  Q}$ is an $(m + 2q - 2,n+2q-1)$-isosymmetric multioperators  under suitable conditions. Moreover, we give some results about
the joint approximate spectrum of an $(m,n)$-isosymmetric multioperators.
\end{abstract} \maketitle

\section{Introduction and preliminaries}

We set below the notations used throughout this paper.
Let ${\mathcal B}({\mathcal H} )$ be the algebra of bounded linear operators on a separable complex
Hilbert space ${\mathcal  H}$. We use the notations $\mathbb{N}$  the set of natural numbers, $\mathbb{N}_0$ the set of  nonnegative integers,
 $\mathbb{R}$ the set of real numbers
 and $\mathbb{C}$ the set of complex
numbers.
An operator $R \in {\mathcal B}({\mathcal H})$  is said to be $m$-isometric operator if
\begin{equation}\label{eq1.1}\sum_{0\leq k \leq m}(-1)^{m-k}\binom{m}{k}R^{*k}R^{k}=0, \end{equation} for some positive integer $m$,
or
\begin{equation}\label{eq1.2}\sum_{0\leq k \leq m}(-1)^{m-k}\binom{m}{k}\|R^{k}x\|^2=0,\;\quad  \forall\;\; x\in { \mathcal{H}}.\end{equation}
 Such
$m$-isometric operators were introduced by J. Agler
and were studied in great detail by  J. Alger and M. Stankus in the papers \cite{AS1,AS2,AS3}.
For more rich theory on $m$-isometric operators and related classes, we invite the most interesting readers  to consult the references \cite{BJZ,CKL1,DK1,DK2,GU1,OA2,OA3,SO4, MP0, YY,FS}.
\par \vskip 0.2 cm \noindent
The concept of $n$-symmetric operators has been introduced and study in \cite{HIL1,SH}. Let $R$ be an operator on  a  Hilbert space,  $R$ is said  to be an  $n$-symmetric operator if $R$ satisfies
\begin{equation}\label{eq1.3}
\sum_{0\leq j\leq n}(-1)^j\binom{n}{j}R^{*j}R^{n-j}=0,
\end{equation}
for some positive integer $n$. It has proved that a power of $n$-symmetric operator is a again $n$-symmetric and the product of two $n$-symmetric operator is also $n$-symmetric under suitable conditions (see  \cite{SH}).\par \vskip 0.2 cm \noindent
\noindent
 Based on  (\ref{eq1.1}) and  (\ref{eq1.3}) the authors in (\cite{st0, St1}) has been introduced the class of  $(m, n)$-isosymmetric operators. An operator $R\in {\mathcal B}({\mathcal H})$ is said to be   $(m, n)$-isosymmetric operator if
\begin{eqnarray*}
 &&\sum_{0\leq j\leq m}(-1)^j\binom{m}{j}R^{*(m-j)}\bigg(\sum_{0\leq k\leq n}(-1)^k\binom{n}{k}R^{*(n-k)}R^k\bigg)R^{m-j}\\&=&
\sum_{0\leq k\leq n}(-1)^k\binom{n}{k}R^{*(n-k)}\bigg(\sum_{0\leq j\leq m}(-1)^j\binom{m}{j}R^{*(m-j)}R^{m-j}\bigg)R^k\\&&=0.
 \end{eqnarray*}
 The study of multioperators has received great interest
 by many authors during recent years. The investigation of multioperators belonging to some specific classes has been quite fashionable since the beginning
of the century, and sometimes it is indeed relevant.  Some developments about these subjects has been done in \cite{AT,BFS1,BFS2,CS,CBS,GR,HMM,HF,OA1, OA4, GBS, SH} and the references therein.
\par \vskip 0.2 cm \noindent
For $d \in \mathbb{N}$, let ${\bf \large R}=(R_1,\cdots,R_d) \in {\mathcal B}^{(d)}(\mathcal{H}):\underbrace{{\mathcal B}(\mathcal{H})\times...\times {\mathcal B}(\mathcal{H})}_{\text{d-times}}$ with $R_j:\mathcal{H}\longrightarrow \mathcal{H}$ be a tuple of  commuting bounded linear operators. Let  $\gamma= (\gamma_1,\cdots,\gamma_d) \in \mathbb{N}_0^d$
and set  $|\gamma|: = \displaystyle\sum_{1\leq j \leq d}\gamma_j$ and $\gamma!: =\displaystyle\prod_{1\leq k\leq d}\gamma_k!$.  Further, define ${\bf \large
R}^\gamma:= R_1^{\gamma_1} R_2^{\gamma_2}\cdots R_d^{\gamma_d}$ where $R^{\gamma_j}=\underbrace{R_j.R_j\cdots R_j}_{\gamma_j-\text{times}}$ $(1\leq j\leq d)$
and ${\bf \large R}^\ast=(R_1^\ast,\cdots,R_d^\ast)$.\par \vskip 0.2 cm \noindent

Let ${\bf\large R}=(R_1,\cdots,R_d)\in {\mathcal B}^{(d)}({\mathcal H})$  be a commuting multioperators and set for $l\in \mathbb{N}_0$:

\begin{equation}\label{eq1.4}{\bf\large S}_{l}({\bf\large R})=\sum_{0\leq k\leq l}(-1)^{l-j}\binom{l}{k}\big(R_1^*+\cdots+R_d^*\big)^k\big(R_1+\cdots+R_d\big)^{l-k},\end{equation}
and
\begin{equation}\label{eq1.5}{\bf\large M}_l({\bf\large R})=\sum_{0.\leq k\leq l}(-1)^{l-k}\binom{l}{k}\bigg(\sum_{|\gamma|=k}\frac{k!}{\gamma!}{\bf\large R}^{*\gamma}{\bf\large R}^\gamma \bigg),\end{equation} we have
$M_0({\bf\large R})=I$ and $M_1({\bf\large R})= \displaystyle\sum_{1\leq j \leq d} R^*_j R_j -I$.\par \vskip 0.2 cm \noindent
Gleason and Richter \cite{GR} considered the multivariable setting of $m$-isometries and studied their properties.
A commuting $d$-tuple of operators  ${\bf \large R}=(R_1,\cdots,R_d)\in {\mathcal B}^{(d)}({\mathcal H})$ is said to be an $m$-isometric multioperators
if it satisfies the operator equation
\begin{equation}
\sum_{0\leq k \leq m}(-1)^{m-k}\binom{m}{k}\bigg(\sum_{|\gamma|=k}\frac{k!}{\gamma!}{\bf \large R}^{\ast \gamma}{\bf \large R}^\gamma \bigg)=0.
\end{equation}
The authors in \cite{MCMN} considered the multivariable setting of $n$-isomsymetries and studied their properties.
A commuting $d$-tuple of operators  ${\bf \large R}=(R_1,\cdots,R_d)\in {\mathcal B}^{(d)}({\mathcal H})$ is said to be an $n$-symmetric multioperators
if it satisfies the operator equation
\begin{equation}
\sum_{0\leq k \leq m}(-1)^{m-k}\binom{m}{k}\big(R_1^*+\cdots+R_d^*\big)^k\big(R_1+\cdots+R_d)^{m-k}=0.\end{equation}
It hold
\begin{equation}\label{eq1.6}{\bf \large M}_{l+1}({\bf \large R})=\sum_{1\leq
j\leq d}R_j^\ast{\bf \large  M}_{l}({\bf \large R})R_j-{\bf \large M }_{l}({\bf R}).\end{equation} In particular if ${\bf \large R}$ is  an $m$-isometric multioperator, then ${\bf \large R}$ is  an $(m+k)$-isometric multioperators
 for all $k\geq 0$.

 Similarly,
 \begin{equation}\label{eq1.7}
 {\bf\large S}_{l+1}({\bf\large R})=\bigg(\sum_{1\leq k\leq d}R_k^*\bigg){\bf\large S}_{l}({\bf\large R})-{\bf\large S}_{l}({\bf\large R})\bigg(\sum_{1\leq k\leq d}R_k\bigg),
 \end{equation}
 and moreover if  ${\bf \large R}$ is  an $n$-symmetric multioperator, then ${\bf \large R}$ is  an $(m+k)$-symmmetric multioperators
 for all $k\geq 0$.
\medskip
\section{ Basic Properties of $(A;(m, n))$-isosymmetries}
\label{S3}
The aim of this section is to initiate the study of
$(m, \;n)$-isosymmetric  multioperators which are classes of operators that contains
$n$-symmetric multioperators and  $m$-isometric multioperators. We give some properties of these classes of operators.

\begin{definition}
Let  ${\bf\large R}=(R_1,\cdots,R_d)\in {\mathcal B}^{(d)}({\mathcal H})$ be a commuting multioperators. ${\bf\large R}$ is said to be  $(m,n)$-isosymmetric if
$\Lambda_{m,\;n}({\bf\large R})=0,$ where
\begin{eqnarray*}\Lambda_{m,\;n}({\bf\large R})&=&
                                                   \displaystyle\sum_{0\leq k\leq n}(-1)^{n-k}\binom{n}{k}\big(R_1^*+\cdots+R_d^*\big)^k{\bf \large M}_m\big({\bf\large R}\big)\big(R_1+\cdots+R_d\big)^{n-k} \\&=&
                                                  \displaystyle \sum_{0\leq k\leq m}(-1)^{m-k}\binom{m}{k}\bigg(\sum_{|\gamma|=k}\frac{k!}{\gamma!}{\bf\large R}^{*\gamma}{\bf \large S}_{n}\big({\bf\large R}\big){\bf\large R}^\gamma \bigg).
                                                 \end{eqnarray*}
\end{definition}
\begin{remark}
$(1)$ Every $m$-isometric multioperators is an $(m,n)$-isosymmetric multioperators
and every $m$-isosymetric multioperators is an $(m,n)$-isosymmetric multioperators.\par \vskip 0.2 cm \noindent
\end{remark}
\begin{remark} We make the following observations\par \vskip 0.2 cm \noindent
\begin{equation}\label{eq2.1}
\Lambda_{1,\;0}\big({\bf\large R} \big)=\sum_{1\leq k\leq d}R_k^*R_k-I,
\end{equation}
\begin{equation}\label{eq2.2}
\Lambda_{0,\;1}\big({\bf\large R} \big)=\sum_{1\leq k\leq d}\big(R_k^*-R_k\big),
\end{equation}
\begin{eqnarray}\label{eq2.3}
\Lambda_{1,\;1}\big({\bf\large R} \big)=\!\big(\sum_{1\leq k\leq d}R_k^*\big)\bigg(\sum_{1\leq j\leq d}R_j^*R_j\!-\!I\bigg)\!-\!\big(\sum_{1\leq j\leq d}R_j^*R_j-I\bigg)\big(\sum_{1\leq k\leq d}R_k\big)\end{eqnarray}
or
\begin{eqnarray}\label{eq2.4}\Lambda_{1,\;1}\big({\bf\large R} \big)=
\sum_{1\leq k\leq d}\bigg(R_k^* \sum_{1\leq j\leq d}\big(R_j^*-R_j\big)R_k\bigg)- \sum_{1\leq j\leq d}\big(R_j^*-R_j\big).
\end{eqnarray}
\end{remark}
\begin{example}
Let  $R\in \mathcal{B}(\mathcal{H})$ be an $(m,n)$-isosymmetric single
operator, $d \in \mathbb{N}$ and
$\beta=(\beta_1,\cdots,\beta_d)\in (\mathbb{R}^d,\;
\|.\|_2)$ with  $\|\beta\|_2^2=\displaystyle\sum_{1\leq j\leq
d}\beta_j^2=1.$ Then the multioperator ${\bf\large
R}=(R_1,\cdots,R_d)$ where $R_j=\beta_jR$ for $j=1,\cdots,d$ is an
 $(m,n)$-isosymmetric multioperators.\par \vskip
0.2 cm \noindent  In fact, it is obvious that $R_lR_k=R_kR_j$ for all
$1\leq j, \;k\leq d.$ From the multinomial expansion, we get
for any natural number $q$
\begin{eqnarray*} 1 = \bigg(\beta_1^2+\beta_2^2+\cdots+\beta_d^2\bigg)^q
&=&\sum_{\gamma_1+\gamma_2+\cdots+\gamma_d=q}\binom{q}{\gamma_1,\gamma_2,\cdots,\gamma_p}\prod_{1\leq
l\leq d}\beta_l^{2\gamma_i}
\\&=&\sum_{|\gamma|=q}\frac{q!}{\gamma!}|\beta^\gamma|^2.\end{eqnarray*}
Thus, we have
\begin{eqnarray*}
{\bf \large M}_m({\bf\large R})&=&\sum_{0\leq j\leq
m}(-1)^{m-j}\binom{m}{j}\bigg(\sum_{|\gamma|=j}\frac{j!}{\gamma!}{\bf\large                             
R}^{*\gamma}{\bf\large R}^\gamma\bigg)\\&=&\sum_{0\leq j\leq
m}(-1)^{m-j}\binom{m}{j}\bigg(\sum_{|\gamma|=j}\frac{j!}{\gamma!}
\prod_{1\leq
j \leq d}\beta_j^{2\gamma_j}
R^{*|\gamma|}R^{|\gamma|}\bigg)\\&=&
\sum_{0\leq j\leq
m}(-1)^{m-j}\binom{m}{j}R^{*j}R^{j}.
\end{eqnarray*}
\begin{eqnarray*}
\Lambda_{m,\;n}({\bf\large R})&=&\displaystyle\sum_{0\leq k\leq n}(-1)^{n-k}\binom{n}{k}\big(R_1^*+\cdots+R_d^*\big)^k{\bf \large M}_m\big({\bf\large R}\big)\big(R_1+\cdots+R_d\big)^{n-k} \\
&=&\bigg(\sum_{1\leq j\leq d}\beta_j\bigg)^n\bigg(     \sum_{0\leq k\leq n}(-1)^{n-k}\binom{n}{k}R^{*(k)}\bigg(\sum_{0\leq j\leq m}(-1)^j\binom{m}{j}R^{*(m-j)}R^{m-j}\bigg)R^{n-k} \bigg)\\&=&0.
\end{eqnarray*}
Therefore ${\bf \large R}$ is a
$(m,n)$-isosymmetric multioperators as required.
\end{example}
\par \vskip 0.2 cm \noindent
In the following  example we show that there is a multioperators which is $(m,n)$-isosymmetric, but neither
 $m$-isometric  nor $n$-isosymmetric multioperators for some  multiindex  and a positive integer $m$ and $n$.
Thus, the proposed new class of multioperators contains the classes of $m$-isometric mutioperatros and $n$-symmetric multioperators
as proper subsets.
\begin{example}
Consider ${\bf\large R}=\big(R_1,\;R_2\big)$ where $R_1=\left(
                                                          \begin{array}{ccc}
                                                           0& 0 & 0 \\
                                                            1 & 0 & 0 \\
                                                           0 & 0 & 0 \\
                                                          \end{array}
                                                        \right)$ and  $R_2=\left(
                                                          \begin{array}{ccc}
                                                           1& 0 & 0 \\
                                                            0 & 1 & 0 \\
                                                           0 & 0 & 1 \\
                                                          \end{array}
                                                        \right)$. A simple computation shows that
                                                         $$\big(R_1+R_2)^*\bigg(  R_1^*R_1+R_2^*R_2-I  \bigg)-\bigg(  R_1^*R_1+R_2^*R_2-I  \bigg)\big(R_1+R_2\big)=0.$$ Therefore, ${\bf\large R}$ is a $(1,1)$-isosymmetric pairs of operators. However
                                                         ${\bf\large R}$ is  not a $1$-isometric and not a 1-symmetric pairs due to the following facts
                                                        $$ R_1^*R_1+R_2^*R_2-I\not=0\; \;\hbox{and}\;\;\big(R_1+R_2)^*-\big(R_1+R_2)\not=0.$$

\end{example}

\begin{theorem}
Let ${\bf\large R}=(R_1,\cdots,R_d)\in {\mathcal B}^{(d)}({\mathcal H})$ be a commuting multioperators. Then following statements hold.
\begin{equation}\label{eq2.5}
\Lambda_{m+1,\;n} \big({\bf\large R})\big)=\sum_{1\leq j\leq d}R_j^*\Lambda_{m,\;n}\big( {\bf\large R}\big)R_j-\Lambda_{m,\;n}\big({\bf\large R}\big).
\end{equation}
\begin{equation}\label{eq2.6}
\Lambda_{m,\;n+1} \big({\bf\large R})\big)=\sum_{1\leq j\leq d}R_j^*\Lambda_{m,\;n}\big( {\bf\large R}\big)-\sum_{1\leq j\leq d}\Lambda_{m,\;n}\big({\bf\large R}\big)R_j.
\end{equation}
\end{theorem}
\begin{proof}
By taking into account (\ref{eq1.6}) we have
\begin{eqnarray*}
\Lambda_{m+1,\;n}({\bf\large R})&=&\displaystyle\sum_{0\leq k\leq n}(-1)^{n-k}\binom{n}{k}\big(R_1^*+\cdots+R_d^*\big)^k{\bf \large M}_{m+1}\big({\bf\large R}\big)\big(R_1+\cdots+R_d\big)^{n-k}\\&=&
\sum_{0\leq k\leq n}(-1)^{n-k}\binom{n}{k}\big(R_1^*+\cdots+R_d^*\big)^k\bigg(  \sum_{1\leq
j\leq d}R_j^\ast{\bf \large  M}_{m}({\bf \large R})R_j-{\bf \large M }_{m}({\bf R}) \bigg)\big(R_1+\cdots+R_d\big)^k
\\&=&\sum_{1\leq j\leq d}R_j^*\sum_{0\leq k\leq n}(-1)^{n-k}\binom{n}{k}\big(R_1^*+\cdots+R_d^*\big)^{k}{\bf\large M}_m({\bf\large R})\big(R_1+\cdots+R_d\big)^{n-k} R_j\\&&-
\sum_{0\leq k\leq n}(-1)^{n-k}\binom{n}{k}\big(R_1^*+\cdots+R_d^*\big)^k{\bf\large M}_m({\bf\large R})\big(R_1+\cdots+R_d\big)^{n-k}\\
&=&\sum_{1\leq j\leq d}R_j^*\Lambda_{m,\;n}({\bf\large R})R_j-\Lambda_{m,\;n}(\bf\large R).
\end{eqnarray*}
\noindent Using a similar argument to that employed above we can prove the identity (\ref{eq2.6}).
\end{proof}
\begin{corollary}\label{cor2.1}
  Let ${\bf\large R}=(R_1,\cdots,R_d)\in {\mathcal B}^{(d)}({\mathcal H})$ be a  commuting multioperators. If ${\bf\large R}$
 is an $(m, n)$-isosymmetric, then ${\bf\large R}$ is $(m^\prime; n^\prime)$-isosymmetric for all
$n^\prime \geq n$  and $m^\prime \geq m.$
\end{corollary}

\begin{lemma}(\cite{CBS})
Let  $\alpha = (\alpha_1, \cdots, \alpha_p),  \gamma= (\gamma_1, \cdots, \gamma_d) \in  {\mathbb N}_0^d$,  $ k \in {\mathbb N}_0$ and $ n \in {\mathbb N}$  be    such that $|\alpha| + | \gamma| + k = n+1$. For $ 1 \leq i \leq d$, let $\epsilon(i) \in {\mathbb N}_0^d$
where $\epsilon(i)$ is the $d$-tuple with $1$ in the $i$-th entry and zeros elsewhere.
   Then,
\begin{equation}\label{eq2.7}
\binom{n+1}{\alpha,\gamma,k} = \sum_{1\leq i \leq d}\bigg(\binom{n}{\alpha-\epsilon(i),\gamma,k} + \binom{n}{\alpha,\gamma-\epsilon(i),k}\bigg) +  \binom{n}{\alpha,\gamma,k-1}.
\end{equation}
\end{lemma}
\begin{proposition}\label{pro2.1}
Let ${\large \bf R}=(R_1,\cdots,R_d) $ and ${\large \bf Q}=(Q_1,\cdots,Q_d)$ be two commuting multioperators for which  $\big[R_j, Q_i\big]=\big[R_j, Q_i^*\big]=0$ for all $j,i\in \{1,\cdots,d\;\}$. Then, for a positive integers $m$ and $n$,  the following identity holds:
\begin{equation}\label{eq2.8}
  {\Lambda}_{m,\;n}({\bf \large R} + {\bf \large Q})=\sum_{0\leq j\leq n}\sum_{|\alpha|+|\gamma|+k =m}\binom{n}{j}\binom{m}{\alpha,\gamma,k}
({\bf \large R+Q})^{*\alpha} \cdot {\bf \large Q}^{*\gamma} \cdot {\Lambda}_{k,\;n-j}({\bf R}){\bf\large S}_{j}({\bf\large Q}){\bf\large R}^\gamma{\bf\large Q}^\alpha, \end{equation} where
$\displaystyle\binom{n}{j}=\displaystyle\frac{n!}{(n-j)!j!}$\;\;\hbox{and}\;\;$\displaystyle\binom{m}{\alpha,\gamma,k}=\displaystyle\frac{m!}{\alpha! \gamma! k!}.$
\end{proposition}
\begin{proof} We prove (\ref{eq2.8}) by two-dimensional induction principle on $(m,n) \in \mathbb{N}^2$.
We first check that (\ref{eq2.8})  is true for $(m,n)=(1,1)$. In fact,
\begin{eqnarray*}
&&\sum_{0\leq j\leq 1}\sum_{|\alpha|+|\gamma|+k =1}\binom{1}{j}\binom{1}{\alpha,\gamma,k}
({\bf \large R+Q})^{*\alpha} \cdot {\bf \large Q}^{*\gamma} \cdot {\Lambda}_{k,\;1-j}({\bf R}){\bf\large S}_{j}({\bf\large Q}){\bf\large R}^\alpha{\bf\large Q}^\gamma\\&=&\sum_{0\leq j\leq 1}\bigg\{\sum_{1\leq i\leq d }\big(R_i^*+Q_i^*\big){\Lambda}_{0,\;1-j}({\bf R}){\bf\large S}_{j}({\bf\large Q})Q_i
+\sum_{1\leq i\leq d}Q_i^*{\Lambda}_{0,\;1-j}({\bf R}){\bf\large S}_{j}({\bf\large Q})R_i\\&&+{\Lambda}_{1,\;1-j}({\bf R}){\bf\large S}_{j}({\bf\large Q})\bigg\}\\&=&
\bigg\{\sum_{1\leq i\leq d }\big(R_i^*+Q_i^*\big){\Lambda}_{0,\;1}({\bf R}){\bf\large S}_{0}({\bf\large Q})Q_i
+\sum_{1\leq i\leq d}Q_i^*{\Lambda}_{0,\;1}({\bf R}){\bf\large S}_{0}({\bf\large Q})R_i+{\Lambda}_{1,\;1}({\bf R}){\bf\large S}_{0}({\bf\large Q})\bigg\}\\&&+\bigg\{\sum_{1\leq i\leq d }\big(R_i^*+Q_i^*\big){\Lambda}_{0,\;0}({\bf R}){\bf\large S}_{1}({\bf\large Q})Q_i
+\sum_{1\leq i\leq d}Q_i^*{\Lambda}_{0,\;0}({\bf R}){\bf\large S}_{1}({\bf\large Q})R_i+{\Lambda}_{1,\;0}({\bf R}){\bf\large S}_{1}({\bf\large Q})\bigg\}
\\&=&\sum_{1\leq i\leq d}\big(R_i^*+Q_i^*\big)\Lambda_{0,\;1}\big({\bf\large R+Q}\big)Q_i+\sum_{1\leq i\leq d}+Q_i^*\Lambda_{0,\;1}\big({\bf\large R+Q}\big)R_i\\&&+\Lambda_{1,\;1}\big({\bf\large R}\big)+\Lambda_{1,\;0}\big({\bf \large R}\big){\bf\large S}_1({\bf\large Q})
\end{eqnarray*}
Based on  (\ref{eq2.1}),(\ref{eq2.2}) and (\ref{eq2.5}) we have
\begin{eqnarray*}
\Lambda_{1,\;1}\big({\bf\large R}\big)+\Lambda_{1,\;0}\big({\bf \large R}\big){\bf\large S}_1({\bf\large Q})&=&\sum_{1\leq i\leq d}R_i^*\Lambda_{0,\;1}
\big({\bf \large R} \big)R_i-\Lambda_{0,\;1}\big({\bf \large R} \big)+ \bigg(\sum_{1\leq i\leq d}R_i^*R_i-I\bigg){\bf\large S}_1\big({\bf\large Q}\big)
\\&=&\sum_{1\leq i\leq d}R_i^*\Lambda_{0,\;1}
\big({\bf \large R} \big)R_i-\Lambda_{0,\;1}\big({\bf \large R} \big)+ \sum_{1\leq i\leq d}R_i^*R_i{\bf\large S}_1\big({\bf\large Q}\big)-{\bf\large S}_1\big({\bf\large Q}\big)
\\&=& \sum_{1\leq i\leq d}R_i^*\Lambda_{0,\;1}
\big({\bf \large R} \big)R_i+ \sum_{1\leq i\leq d}R_i^*{\bf\large S}_1\big({\bf\large Q}\big)R_i-\Lambda_{0,\;1}\big({\bf \large R} \big)-{\bf\large S}_1\big({\bf\large Q}\big)
\\&=&\sum_{1\leq i\leq d}R_i^*\big(\Lambda_{0,\;1}\big({\bf\large R}\big)+{\bf\large S}_1\big({\bf\large Q}\big)\big)R_i-\bigg(\Lambda_{0,\;1}\big({\bf \large R} \big)+{\bf\large S}_1\big({\bf\large Q}\big)\bigg)
\\&=& \sum_{1\leq i\leq d}R_i^*\Lambda_{0,\;1}\big({\bf\large R+Q} \big)R_i-\Lambda_{0,\;1}\big({\bf\large R+Q} \big).
\end{eqnarray*}
\begin{eqnarray*}
&&\sum_{0\leq j\leq 1}\sum_{|\alpha|+|\gamma|+k =1}\binom{1}{j}\binom{1}{\alpha,\gamma,k}
({\bf \large R+Q})^{*\alpha} \cdot {\bf \large Q}^{*\gamma} \cdot {\Lambda}_{k,\;1-j}({\bf R}){\bf\large S}_{j}({\bf\large Q}){\bf\large R}^\alpha{\bf\large Q}^\gamma\\&=&\sum_{1\leq i\leq d}\big(R_i^*+Q_i^*\big)\Lambda_{0,\;1}\big({\bf\large R+Q}\big)Q_i+\sum_{1\leq i\leq d}+Q_i^*\Lambda_{0,\;1}\big({\bf\large R+Q}\big)R_i\\&&+ \sum_{1\leq i\leq d}R_i^*\Lambda_{0,\;1}\big({\bf\large R+Q} \big)R_i-\Lambda_{0,\;1}\big({\bf\large R+Q} \big)
\\&=&
\sum_{1\leq i\leq d}\big(R_i^*+Q_i^*\big)\Lambda_{0,\;1}\big({\bf\large R+Q}\big)Q_i+\sum_{1\leq i\leq d}+\big(R_i^*+Q_i^*\big)\Lambda_{0,\;1}\big({\bf\large R+Q}\big)R_i-\Lambda_{0,\;1}\big({\bf\large R+Q} \big)\\&=&
\sum_{1\leq i\leq d}\big(R_i^*+Q_i^*\big)\Lambda_{0,\;1}\big({\bf\large R+Q}\big)\big(R_i+Q_i\big)-\Lambda_{0,\;1}\big({\bf\large R+Q} \big)\\&=&
\Lambda_{1,\;1}\big({\bf\large R+Q}\big) \quad \quad (\hbox{by}\;\; (\ref{eq2.5})).
 \end{eqnarray*}
So the identity (\ref{eq2.8})  holds for $(m,n)=(1,1).$ Assume that the identity (\ref{eq2.8})  holds for $(m,1)$ and prove it for $(m+1,1).$
According to (\ref{eq2.5}) and the induction hypothesis we have
\begin{eqnarray*}
&&\Lambda_{m+1,\;1} \big({\bf\large R+Q}\big)=\sum_{1\leq i\leq d}\big(R_i^*+Q_i^*\big)\Lambda_{m,\;1}\big( {\bf\large R+Q}\big)\big(R_i+Q_i\big)-\Lambda_{m,\;1}\big({\bf\large R+Q}\big).
\\&=&\sum_{1\leq i\leq d}\big(R_i^*+Q_i\big)\bigg\{  \sum_{0\leq j\leq 1}\sum_{|\alpha|+|\gamma|+k =m}\binom{1}{j}\binom{m}{\alpha,\gamma,k}
({\bf \large R+Q})^{*\alpha} \cdot {\bf \large Q}^{*\gamma} \cdot {\Lambda}_{k,\;1-j}({\bf R}){\bf\large S}_{j}({\bf\large Q}){\bf\large R}^\gamma{\bf\large Q}^\alpha  \bigg\}\big(R_i+Q_i\big)\\&&-
\sum_{0\leq j\leq 1}\sum_{|\alpha|+|\gamma|+k =m}\binom{1}{j}\binom{m}{\alpha,\gamma,k}
({\bf \large R+Q})^{*\alpha} \cdot {\bf \large Q}^{*\gamma} \cdot {\Lambda}_{k,\;1-j}({\bf R}){\bf\large S}_{j}({\bf\large Q}){\bf\large R}^\gamma{\bf\large Q}^\alpha\\&=&
\sum_{0\leq j\leq 1}\sum_{|\alpha|+|\gamma|+k =m}\binom{1}{j}\binom{m}{\alpha,\gamma,k}
({\bf \large R+Q})^{*\alpha} \cdot {\bf \large Q}^{*\gamma} \cdot\bigg[\sum_{1\leq i\leq d} R_i^*{\Lambda}_{k,\;1-j}({\bf R})R_i- {\Lambda}_{k,\;1-j}({\bf R})\\&&+\sum_{1\leq i\le d}\bigg\{ R_i^*{\Lambda}_{k,\;1-j}({\bf R})Q_i+Q_i^*{\Lambda}_{k,\;1-j}({\bf R})R_i+Q_i^*{\Lambda}_{k,\;1-j}({\bf R})Q_i  \bigg\}\bigg]{\bf\large S}_{j}({\bf\large Q}){\bf\large R}^\gamma{\bf\large Q}^\alpha\\&=&
\sum_{0\leq j\leq 1}\sum_{|\alpha|+|\gamma|+k =m}\binom{1}{j}\binom{m}{\alpha,\gamma,k}
({\bf \large R+Q})^{*\alpha} \cdot {\bf \large Q}^{*\gamma} \cdot{\Lambda}_{k+1,\;1-j}({\bf R}){\bf\large S}_{j}({\bf\large Q}){\bf\large R}^\gamma{\bf\large Q}^\alpha\\&&+\sum_{0\leq j\leq 1}\sum_{|\alpha|+|\gamma|+k =m}\binom{1}{j}\binom{m}{\alpha,\gamma,k}
({\bf \large R+Q})^{*\alpha} \cdot {\bf \large Q}^{*\gamma}\bigg\{\\&& \big(R_i^*+Q_i^*\big){\Lambda}_{k,\;1-j}({\bf R})Q_i+Q_i^*{\Lambda}_{k,\;1-j}({\bf R})R_i \bigg\}\bigg]{\bf\large S}_{j}({\bf\large Q})\bigg\}{\bf\large R}^\gamma{\bf\large Q}^\alpha\\&=&
\sum_{0\leq j\leq 1}\sum_{|\alpha|+|\gamma|+k =m}\binom{1}{j}\binom{m}{\alpha,\gamma,k}
({\bf \large R+Q})^{*\alpha} \cdot {\bf \large Q}^{*\gamma} \cdot{\Lambda}_{k+1,\;1-j}({\bf R}){\bf\large S}_{j}({\bf\large Q}){\bf\large R}^\gamma{\bf\large Q}^\alpha\\&&+\sum_{0\leq j\leq 1}\sum_{|\alpha|+|\gamma|+k =m}\binom{1}{j}\binom{m}{\alpha,\gamma,k}
\sum_{1\leq i\leq d}({\bf \large R+Q})^{*\alpha} \big(R_i^*+Q_i^*\big)\cdot {\bf \large Q}^{*\gamma}{\Lambda}_{k,\;1-j}({\bf R})Q_i{\bf\large S}_{j}({\bf\large Q}){\bf\large R}^\gamma{\bf\large Q}^\alpha\\&&+
\sum_{0\leq j\leq 1}\sum_{|\alpha|+|\gamma|+k =m}\binom{1}{j}\binom{m}{\alpha,\gamma,k}\sum_{1\leq i\leq d}
({\bf \large R+Q})^{*\alpha} \cdot {\bf \large Q}^{*\gamma}
Q_i^*{\Lambda}_{k,\;1-j}({\bf R})R_i {\bf\large S}_{j}({\bf\large Q}){\bf\large R}^\gamma{\bf\large Q}^\alpha
\\&=&\sum_{0\leq j\leq 1}\sum_{|\alpha|+|\gamma|+k=m+1}\binom{1}{j}\binom{m+1}{\alpha,\gamma,k}({\bf \large R+Q})^{*\alpha} \cdot {\bf \large Q}^{*\gamma} \cdot {\Lambda}_{k,\;1-j}({\bf R}){\bf\large S}_{j}({\bf\large Q}){\bf\large R}^\gamma{\bf\large Q}^\alpha.
\end{eqnarray*}
So (\ref{eq2.8}) holds for $(m, 1)$ implies (\ref{eq2.8}) holds for $(m+1,1)$.\par \vskip 0.2 cm \noindent
Now assume that (\ref{eq2.8}) holds for $(m,n)$ and we prove that (\ref{eq2.8}) holds for $(m,n+1)$.\par \vskip 0.2 cm \noindent
According to (\ref{eq2.6}) and the induction hypothesis  we have
\begin{eqnarray*}
&&\Lambda_{m,\;n+1} \big({\bf\large R+Q}\big)\\&=&\sum_{1\leq i\leq d}\big(R_i^*+Q_i^*\big)\Lambda_{m,\;n}\big( {\bf\large R+Q}\big)-\sum_{1\leq i\leq d}\Lambda_{m,\;n}\big({\bf\large R+Q}\big)\big(R_i+Q_i\big)\\&=&
\sum_{1\leq i\leq d}\big(R_i^*+Q_i^*\big)\bigg[  \sum_{0\leq j\leq n}\sum_{|\alpha|+|\gamma|+k =m}\binom{n}{j}\binom{m}{\alpha,\gamma,k}
({\bf \large R+Q})^{*\alpha} \cdot {\bf \large Q}^{*\gamma} \cdot {\Lambda}_{k,\;n-j}({\bf R}){\bf\large S}_{j}({\bf\large Q}){\bf\large R}^\gamma{\bf\large Q}^\alpha \bigg]\\\end{eqnarray*}
\begin{eqnarray*}&&-
\sum_{1\leq i\leq d}\bigg[\sum_{0\leq j\leq n}\sum_{|\alpha|+|\gamma|+k =m}\binom{n}{j}\binom{m}{\alpha,\gamma,k}
({\bf \large R+Q})^{*\alpha} \cdot {\bf \large Q}^{*\gamma} \cdot {\Lambda}_{k,\;n-j}({\bf R}){\bf\large S}_{j}({\bf\large Q}){\bf\large R}^\gamma{\bf\large Q}^\alpha\bigg]\big(R_i+Q_i\big)\\&=&
\sum_{|\alpha|+|\gamma|+k =m}\binom{m}{\alpha,\gamma,k}({\bf \large R+Q})^{*\alpha} \cdot {\bf \large Q}^{*\gamma}\bigg[\mathbf{A}(m,n,{\bf\large R},\;{\bf\large Q}\big)\bigg]{\bf\large R}^\gamma{\bf\large Q}^\alpha,
\end{eqnarray*}
where
$$\mathbf{A}(m,n,{\bf\large R},\;{\bf\large Q}\big)=\sum_{0\leq j\leq n}\binom{n}{j}\sum_{1\leq i\leq d}\big(R_i^*+Q_i^*\big){\Lambda}_{k,\;n-j}({\bf R}){\bf\large S}_{j}({\bf\large Q})-\sum_{0\leq j\leq n}\binom{n}{j}
{\Lambda}_{k,\;n-j}({\bf R}){\bf\large S}_{j}({\bf\large Q})\sum_{1\leq i\leq d}\big(R_i+Q_i\big).$$
Based on (\ref{eq1.7}) and (\ref{eq2.6}) we have
\begin{eqnarray*}
\mathbf{A}(m,n,{\bf\large R},\;{\bf\large Q}\big)&=&\sum_{0\leq j\leq n}\binom{n}{j}\bigg\{ \sum_{1\leq i\leq d}R_i^*{\Lambda}_{k,\;n-j}({\bf R})-{\Lambda}_{k,\;n-j}({\bf R})\big(\sum_{1\leq i\leq d}R_i\big)\bigg\}{\bf\large S}_{j}({\bf\large Q})\\&&+
\sum_{0\leq j\leq n}\binom{n}{j}{\Lambda}_{k,\;n-j}({\bf R})\bigg\{ \big(\sum_{1\leq i\leq d}Q_i^*\big){\bf\large S}_{j}({\bf\large Q})-{\bf\large S}_{j}({\bf\large Q})\big(\sum_{1\leq i\leq d}Q_i\big)\bigg\}\\&=&
\sum_{0\leq j\leq n}\binom{n}{j}{\Lambda}_{k,\;n+1-j}({\bf R}){\bf\large S}_{j}({\bf\large Q})+   \sum_{0\leq j\leq n}\binom{n}{j}{\Lambda}_{k,\;n-j}({\bf R})   {\bf\large S}_{j+1}({\bf\large Q})\\&=&
{\Lambda}_{k,\;n+1}({\bf\large R})+\sum_{1\leq j\leq n}\bigg(\binom{n}{j}+\binom{n}{j-1}\bigg){\Lambda}_{k,\;n+1-j}({\bf\large R}){\bf\large S}_{j}({\bf\large Q})+
{\Lambda}_{k,\;0}({\bf\large R}){\bf\large S}_{n+1}({\bf\large Q})
\\&=&\sum_{0\leq j\leq n+1}\binom{n+1}{j}{\Lambda}_{k,\;n+1-j}({\bf R}){\bf\large S}_{j}({\bf\large Q}).
\end{eqnarray*}
By using the above relation we obtain
  $${\Lambda}_{m,\;n+1}({\bf \large R} + {\bf \large Q})=\sum_{0\leq j\leq n+1}\sum_{|\alpha|+|\gamma|+k =m}\binom{n+1}{j}\binom{m}{\alpha,\gamma,k}
({\bf \large R+Q})^{*\alpha} \cdot {\bf \large Q}^{*\gamma} \cdot {\Lambda}_{k,\;n+1-j}({\bf R}){\bf\large S}_{j}({\bf\large Q}){\bf\large R}^\gamma{\bf\large Q}^\alpha. $$ Consequently, the identity (\ref{eq2.8}) holds for all $m$ and $n$.This ends the proof.
\end{proof}
\par \vskip 0.2 cm \noindent Let ${\large \bf Q}=(Q_1,\cdots,Q_d)\in \mathcal{B}^{(d)}({\mathcal H})$ be a commuting multioperators. Recall that ${\large \bf Q}$ is said to be $p$-nilpotent, $p \in \mathbb{N}$,  if
${\bf \large Q}^\alpha=Q_1^{\alpha_1}\cdots Q_d^{\alpha_d}=0$ for all
$\alpha=(\alpha_1,\cdots,\alpha_d) \in \mathbb{N}_0^d$ with $|\alpha|=p$. (See \cite{GU2} ).\par\vskip 0.2cm\noindent
\begin{remark}
Let ${\large \bf Q}=(Q_1,\cdots,Q_d)$ be a commuting multioperators. If ${\large \bf Q}$ is nilpotent of order $k$,then ${\bf\large S}_r\big({\bf\large Q}\big)=0$ for all $r\geq 2k.$
\end{remark}
\begin{theorem}\label{thm1}
Let ${\large \bf R}=(R_1,\cdots,R_d) $ and ${\large \bf Q}=(Q_1,\cdots,Q_d)$ be two commuting multioperators for which  $\big[R_j, Q_i\big]=\big[R_j, Q_i^*\big]=0$ for all $j,i\in \{1,\cdots,d\;\}$. If ${\bf \large{ R}}$ is $(m,n)$-isosymmetric  multioperators and
${\bf \large{ Q}} $ is nilpotent of order $k$, then  ${\bf \large{R + Q}}$ is $(m + 2 k -2, n+2k-1)$-isosymmetric multioperators.
\end{theorem}
\begin{proof} We need to show that $ {\Lambda}_{m+2q-2,\;n+2q-1}({\bf \large R} + {\bf \large Q})=0.$
According to $(\ref{eq2.5})$ we have
 \begin{eqnarray*}&&{\Lambda}_{m+2q-2,\;n+2q-1}({\bf \large R} + {\bf \large Q})=\\&&\sum_{0\leq j\leq n+2q-1}\sum_{|\alpha|+|\gamma|+k =m+2q-2}\binom{n+2q-1}{j}\binom{m+2q-2}{\alpha,\gamma,k}
({\bf \large R+Q})^{*\alpha} \cdot {\bf \large Q}^{*\gamma} \cdot {\Lambda}_{k,\;n+2k-1-j}({\bf R}){\bf\large S}_{j}({\bf\large Q}){\bf\large R}^\gamma{\bf\large Q}^\alpha.\end{eqnarray*}
\rm(i) If  $j\geq 2q$ or $\max\{ |\alpha|,\;|\gamma|\;\}\geq q$ we have  ${\bf\large S}_j\big({\bf\large Q}\big)=0$ or ${\bf\large Q}^{*\gamma}=0$ or
${\bf\large Q}^{\alpha}=0.$ \par \vskip 0.2 cm \noindent \rm(ii) If $j\leq 2q-1$ and  $\max\{ |\alpha|,\;|\gamma|\;\}\leq q-1$ we  have\par \vskip 0.2 cm \noindent
$k=m+2q-2-|\alpha|-|\gamma|\geq m$ and $n+2q-1-j\geq n $  and therefore ${\Lambda}_{k,\;n+2k-1-j}({\bf R})=0$ by Corollary \ref{cor2.1}.\par \vskip 0.2 cm \noindent
By combining \rm(i) and \rm(ii) we can conclude that   ${\Lambda}_{m+2q-2,\;n+2q-1}({\bf R})=0$.
\end{proof}

\par \vskip 0.2 cm \noindent
A particularly interesting consequences of Theorems \ref{thm1} are the following.

\begin{corollary}
Let ${\large \bf R}=(R_1,\dots,R_d)\in {\mathcal B}^{(d)}({\mathcal H})$  be an $(m,n)$-isosymetric commuting multioperators  and let ${\large \bf Q}=(Q_1,\cdots,Q_d)\in {\mathcal B}^{(d)}({\mathcal H})$ be a $q$-nilpotent commuting muoperators. Then\\ ${\bf\large R\otimes}{\bf\large I}+{\bf \large I}\otimes {\bf \large Q} :=(R_1\otimes I+I\otimes Q_1,\cdots ,R_d\otimes I+I\otimes Q_d)\in {\mathcal B}^{(d)}(\mathcal{H}\overline{\otimes}\mathcal{H})$ is an $(m+2q-2,n+2q-1)$-isosymmetric multioperators.
\end{corollary}

\begin{proof}
It is obviously that $\big[(R_k\otimes I), ( I\otimes N_j)\big]=\big[(R_k\otimes I), (I\otimes N_j)^*\big]=0$  for all $j,k=1,\cdots,d$.
Moreover, it is easy  to check that ${\bf \large R}\otimes {\bf \large I}=(R_1\otimes I,\;\cdots,R_d\otimes I)\in {\mathcal B}(\mathcal{H}^{(d)}\overline{\otimes}\mathcal{H})$ is an $(m, n)$-isosymmetric commuting multioperators and
${\bf \large I}\otimes {\bf \large Q}\in {\mathcal B}^{(d)}(\mathcal{H}\overline{\otimes}\mathcal{H})$ is  nilpotent commuting multioperators of order  $q$. Therefore  ${\bf \large R}\otimes {\bf \large I}$ and ${\bf\large I} \otimes {\bf \large Q}$ satisfy the conditions of Theorem \ref{thm1}. Consequently,  ${\bf\large R}{\otimes}{\bf\large I}+{\bf \large I}\otimes {\bf \large Q}$ is an $(m+2q-2,n+2q-1)$-isosymmetric multioperators.
\end{proof}
\begin{corollary} Let ${\bf\large A}=(A_1,\cdots,A_d)\in {\mathcal B}^{(d)}({\mathcal H})$ be an $(m,n)$-isomsymetric multioperators. If
${\bf\large B}=(B_1,\cdots,B_d)\in {\mathcal B}^{(d)}({\mathcal H})$ is defined by
$$B_k=\left(
    \begin{array}{cccc}
      A_k & \mu_k I & 0 &\cdots \\
      0 &  \ddots &  \ddots &\ddots \\
      \ddots &   \ddots &  \ddots  & \mu_kI \\
      0 &  \ddots  & 0& A_k \\
    \end{array}
  \right) \;\text{on}\;\;{\mathcal H}^{(q)}:={\mathcal H}\oplus\cdots \oplus{\mathcal H}
$$
  where $\mu_k \in \mathbb{C}$  for  $k = 1,\cdots, d $, then ${\bf \large B}$ is an
  $(m+2q-2, n+2q-1)$-isosymmetric multioperators.
\end{corollary}

\begin{proof}  Obviously we have
$$B_k=\left(
    \begin{array}{cccc}
      A_k& 0 & 0 &\cdots \\
      0 &  \ddots &  \ddots &\ddots \\
      \ddots &   \ddots &  \ddots  & 0 \\
      0 &  \ddots  & 0& A_k\\
    \end{array}
  \right)
+\left(
    \begin{array}{cccc}
      0& \mu_kI & 0 &\cdots \\
      0 &  \ddots &  \ddots &\ddots \\
      \ddots &   \ddots &  \ddots  & \mu_kI \\
      0 &  \ddots  & 0& 0\\
    \end{array}
  \right)=R_k+Q_k\;\;\text{for}\;\;k=1,\cdots,d,$$ and so we my write
 ${\bf\large B}={\bf\large R}+{\bf \large Q}=(R_1+ Q_1,\cdots,R_d+Q_d).$ By Direct computations, we show that ${\bf\large R}$ is $(m,n)$-isosymmetic multioperators, ${\bf\large Q}$ is $q$-nilpotent and
 $$\big[R_k,Q_j\big]=\big[R_k,Q_j^*\big]=0,\;;\;\hbox{for}\;\;k,j\in \{1,\cdots,d\}.$$
  \noindent  According to Theorem \ref{thm1}, ${\bf \large B}$ is a $(m+2q-2, n+2q-1)$-isosymmetric multioperators.
\end{proof}

In \cite[Theorem 3.1]{BJZ} it has been proved that if $R\in {\mathcal B}({\mathcal H})$ is a strict $m$-isometry, then the list
of operators $\big\{ \;\displaystyle\sum_{0\leq j\leq l}(-1)^l\binom{k}{l}R^{*k-l}R^{k-l}, k = 0, 1,\;\cdots m - 1 \;\big\}$ is linearly
 independent. However in \cite[Theorem 4.1]{SH}  it has been proved that if $R$ is a strict $n$-symmetric operator,  then the list
of operators $$\big\{ \;\displaystyle\sum_{0\leq j\leq l}(-1)^l\binom{k}{l}R^{*l}R^{k-l}, k = 0, 1,\;\cdots n - 1 \;\big\},$$ is linearly
 independent.
\par \vskip 0.2 cm \noindent In the following proposition we extend these results to $m$-isometric and $n$-symmetric multioperatros.
\begin{proposition}
Let  ${\bf\large R}=(R_1,\cdots,R_d)\in {\mathcal B}^{(d)}({\mathcal H})$ be a commuting multioperators such that $\Lambda_{m-1,;n-1}({\bf\large R})\not=0$ for some positive integers $m\geq 2$ and  $n\geq2$. The following properties hold.\par \vskip 0.2 cm \noindent $1)$ If ${\bf\large R}$ is $m$-isometric multioperators,  then the list of operators
$$\{\;\Lambda_{k,\;n-1}({\bf\large R}), k = 0, 1,\;\cdots,m-1 \;\}$$ is linearly independent.
\par \vskip 0.2 cm \noindent $2)$ If ${\bf\large R}$ is $n$-symmetricc multioperators,  then the list of operators
$$\{\;\Lambda_{m-1,\;l}({\bf\large R}), l = 0, 1,\;\cdots,n-1 \;\}$$ is linearly independent.
\end{proposition}

\begin{proof} (1)  The proof  of the statement (1) is essentially
based on the multiple use of the following identity
\begin{equation}\label{eq2.3}
\Lambda_{m+1,\;n} \big({\bf\large R})\big)=\sum_{1\leq j\leq d}R_j^*\Lambda_{m,\;n}\big( {\bf\large R}\big)R_j-\Lambda_{m,\;n}\big({\bf\large R}\big).
\end{equation}
 Assume that
$$\sum_{0\leq k \leq m-1}a_k\Lambda_{k,\;n-1}\big({\bf\large R}\big)=0,$$
for some complex numbers $a_k.$
 Multiplying the above equation on the left by $R_i^*$ and right by $R_i$  for $i=1,\cdots,d $  we obtain the following relation
 $$\sum_{0\leq k \leq m-1}a_k \bigg(\sum_{1\leq i\leq i\leq d}R_i^*\Lambda_{k,\;n-1}\big({\bf\large R}\big)R_i\bigg)=0$$ and subtracting two
equations, we have
$$\sum_{0\leq k\leq m-1}a_k\bigg(\sum_{1\leq i\leq d}R_i^*\Lambda_{k,\;n-1}\big({\bf\large R}\big)R_i-\Lambda_{k,\;n-1}\big({\bf\large R}\big)\bigg)=\sum_{0\leq k \leq m-1}a_k\Lambda_{k+1,\;n-1}\big({\bf\large R}\big)=0$$
From an argument similar to the above applied to the equation $$\displaystyle\sum_{0\leq k \leq m-1}a_k\Lambda_{k+1,\;n-1}\big({\bf\large R}\big)=0,$$ we get
$$\sum_{0\leq k \leq m-1}a_k\Lambda_{k+2,\;n-1}\big({\bf\large R}\big)=0.$$   Following  the same steps, we  can obtain
 $$\sum_{0\leq k \leq m-1}a_k\Lambda_{k+l,\;n-1}\big({\bf\large R}\big)=0\;\;\;\text{for all}\;\;l\in \mathbb{N}.$$
In the case that ${\bf\large R}$ is $m$-isometric multioperators, it follows that   $\Lambda_{j,\;n-1}\big({\bf\large R}\big)=0$ for all $j\geq m$.
 By considering   the following cases, we get
\par \vskip 0.2 cm \noindent For $l=m-1$,
 $\displaystyle\sum_{0\leq k \leq m-1}a_k\Lambda_{k+l,\;n-1}\big({\bf\large R}\big)=0 \Rightarrow a_0\Lambda_{m-1,\;n-1}\big({\bf\large R}\big)=0,$ So we have that $a_0=0.$ \par \vskip 0.2 cm \noindent
 For $l=m-2$, $\displaystyle\sum_{0\leq k \leq m-1}a_k\Lambda_{k+l,\;n-1}\big({\bf\large R}\big)=0 \Rightarrow a_1\Lambda_{m-1,\;n-1}\big({\bf\large R}\big)=0,$ So we have that $a_1=0.$ \par \vskip 0.2 cm \noindent Continuing this process we see that all $a_k = 0$ for $k=2,\;\cdots,\;m-1$. Hence the result is proved.
 \par \vskip 0.2 cm \noindent (2) The proof is similar to the above one by using the identity
 \begin{equation}\label{eq2.4}
\Lambda_{m,\;n+1} \big({\bf\large R})\big)=\sum_{1\leq j\leq d}R_j^*\Lambda_{m,\;n}\big( {\bf\large R}\big)-\sum_{1\leq j\leq d}\Lambda_{m,\;n}\big({\bf\large R}\big)R_j.
\end{equation}
\end{proof}

\section{spectral properties of $(m,n)$-isosymmetric multioperators}
\

\

\
Let ${\bf \large R}=(R_1,\cdots,R_d) \in {\mathcal B}^{(d)}({\mathcal H})$ be a commuting multioperators.  Following \cite{kkmm}, the authors noted that  \par \vskip 0.2 cm \noindent
{\em (1)} A point $\mu=(\mu_1,\cdots,\mu_d)\in \mathbb{C}^d$
is in the  joint  eigenvalue of ${\bf \large R}$ if there exists a nonzero
vector $u\in \mathcal{H}$ such that
$$\big(R_l-\mu_l\big)u=0\;\;\hbox{for all }\;\;l=1,2,\cdots,d.$$
 The joint point spectrum, denoted by $\sigma_p({\bf \large
R})$  of ${\bf \large R}$ is the set of all joint eigenvalues of
${\bf \large R},$
that is, $$\sigma_p({\bf\large
R})=\big\{\mu=(\mu_1,\cdots,\mu_d) \in \mathbb{C}^d:\;\bigcap_{1 \leq l \leq
d}\mathcal{N}(R_l-\mu_l)\not=\{0\} \big\}.$$
\noindent {\em (2)} A point
$\mu=(\mu_1,\cdots,\mu_d)\in \mathbb{C}^d$ is in
the joint approximate point spectrum $\sigma_{ap}({\bf \large R})$
if and only if there exists a sequence $\{u_k \}_k \subset {\mathcal H}$ such that $\|u_k\|=1$ and
$$\big(R_l -\mu_l \big) u_k \longrightarrow 0  \;\;\hbox{as}
\;\;k\longrightarrow \infty \;\;\hbox{for every }\;\;l=1,\cdots,d.
$$
For additional information on these concepts, see \cite{CM,MC}.\par \vskip 0.2 cm \noindent
In the following results we examine some spectral
properties of an $(m,n)$ isosymmetric commuting multioperators. That
extend the cases of  $m$-isometries and $n$-symmetric multioperators studied in
\cite{GR} and \cite{CS}.\par\vskip 0.2 cm  \noindent We put
$$\mathbb{B}(\mathbb{C}^d):=\{ \mu=(\mu_1,\cdots,\mu_d)\in
\mathbb{C}^d \; \; \|\mu\|_2=\bigg(\sum_{1\leq l \leq
d}|\mu_l|^2\bigg)^{\frac{1}{2}}<1\;\}$$  and
$$\partial\mathbb{B}(\mathbb{C}^d):=\{
\mu=(\mu_1,\cdots,\mu_d)\in \mathbb{C}^d \;/ \;
\|\mu\|_2 =\bigg(\sum_{1\leq l \leq
d}|\mu_l|^2\bigg)^{\frac{1}{2}} =1\;\}$$\par\vskip 0.2 cm
\noindent

In ( \cite{GR}, Lemma 3.2), the authors have proved that if ${\bf \large R}$
is an $m$-isometric multioperators, then the joint approximate point spectrum
of ${\bf\large R}$ is in the boundary of the unit ball
$\mathbb{B}(\mathbb{C}^d)$. However in \cite[Theorem 4.1]{CS}, the authors have proved that if ${\bf \large R}$
is an $n$-symmetric multioperators, then the joint approximate point spectrum
of ${\bf\large R}$ is in the set
$\bigg\{  (\mu_1,\cdots,\mu_d) \in \mathbb{C}^d\;\;/ \;\displaystyle\sum_{1\leq l\leq d}\mu_l \in \mathbb{R}\;\bigg\}.$

\begin{proposition}
Let ${\bf \large R}=(R_1,\cdots,R_d)\in {\mathcal B}^{(d)}({\mathcal H})$ be a commuting multioperators and  $\mu=(\mu_1,\cdots,\mu_d)\in \mathbb{C}^d$. If ${\bf\large R}$ is an  $(m,n)$-isosymmetric and $\displaystyle\sum_{1\leq k\leq d}\mu_k\notin \sigma_{ap}\big( \sum_{1\leq k\leq d}R_k^*\big)$, then the following property
holds
$$ \big[0 \big] \cap \sigma_{ap}({\bf\large R})=\emptyset,$$ where $$\big[0
\big]:=
 \big\{ (\mu_1,\cdots,\mu_d) \in
\mathbb{C}^d:\;\prod_{1\leq l\leq d}\mu_l=0\;\big\}.$$
\par\vskip 0.2 cm
\end{proposition}
\begin{proof} Assume that $ \big[0 \big] \cap \sigma_{ap}({\bf R}) \not=\emptyset$ and  let  $\mu=(\mu_1,\cdots,\mu_d)\in  \big[0 \big] \cap  \sigma_{ap}({\bf \large R})$, then there exists a
sequence $(u_k)_{k\geq 1} \subset \mathcal{H}$, with $||u_k||=1$
such that $\big({ R_l}-\mu_l I)u_k \longrightarrow 0$ for all
$l=1,2,\cdots,d$. Since for $\gamma_l>1$,
$$R_j^{\gamma_l}-\mu_l^{\gamma_l}=\big(R_l-\mu_l\big)\sum_{1\leq k\leq \mu_l}\mu_l^{k-1}R_l^{\mu_l-k}$$ By induction, for  $\gamma=(\gamma_1,\cdots,\gamma_d)
\in \mathbb{N}_0^d$, we have $$ ({\bf\large R}^\gamma-\mu^\gamma I)=
\sum_{1\leq k\leq d}\bigg( \prod_{i\leq
k}\mu_i^{\gamma_i}\bigg)\bigg(
R_k^{\gamma_k}-\mu_k^{\gamma_k}\bigg)\prod_{i>k}R_i^{\gamma_i}.$$
We deduce that
$ \big({\large \bf R}^\gamma-\mu^\gamma I\big)u_k \to 0$ as $k
\longrightarrow \infty$ for all $\gamma=(\gamma_1,\cdots,\gamma_D)\in \mathbb{N}_0^d$.
Since $\Lambda_{m,n}\big({\bf \large R}\big)=0$, it follows that
\begin{eqnarray*}
0&=&\sum_{0\leq j\leq m}(-1)^{m-j}\binom{m}{j}\bigg(\sum_{|\gamma|=j}\frac{j!}{\gamma!}{\bf \large R}^{\ast \gamma}{\bf\large S}_n\big({\bf \large R}\big){\bf \large R}^\gamma\bigg)  u_k\\&=&
{\bf\large S}_n\big({\bf \large R}\big)u_k+
\sum_{1\leq j\leq m}(-1)^{m-j}\binom{m}{j}\bigg(\sum_{|\gamma|=j}\frac{j!}{\gamma!}{\bf \large R}^{\ast \gamma}{\bf\large S}_n\big({\bf \large R}\big){\bf \large R}^\gamma\bigg)  u_k\\&=&
{\bf\large S}_n\big({\bf \large R}\big)u_k+
\sum_{1\leq j\leq m}(-1)^{m-j}\binom{m}{j}\bigg(\sum_{|\gamma|=j}\frac{j!}{\gamma!}{\bf \large R}^{\ast \gamma}{\bf\large S}_n\big({\bf \large R}\big)\big({\bf \large R}^\gamma -\mu^\gamma\big)\bigg)  u_k\\&=&
\end{eqnarray*}
By taking $k\to \infty$ we get $\displaystyle\lim_{k\to \infty}{\bf\large S}_n\big({\bf \large R}\big)u_k=0$

\begin{eqnarray*}
\displaystyle\lim_{k\to \infty}{\bf\large S}_n\big({\bf \large R}\big)u_k=0&\Rightarrow& \lim_{k\to \infty}\bigg( \sum_{0\leq j\leq n}(-1)^j\binom{n}{j}\big(R_1^*+\cdots +R_d^*\big)^j\big(R_1+\cdots+R_d\big)^{n-j}u_k   \bigg)=0
\\&\Rightarrow&\lim_{k\to \infty}\big(\mu_1+\cdots+\mu_d-R_1^*-\cdots-R_d^*\big)^nu_k=0.
\end{eqnarray*}

If $\big(\mu_1+\cdots+\mu_d-R_1^*-\cdots-R_d^*\big)$ is bounded from below, then so is  $\big(\mu_1+\cdots+\mu_d-R_1^*-\cdots-R_d^*\big)^n$ is bounded from below
and therefore
$$\|\big(\mu_1+\cdots+\mu_d-R_1^*-\cdots-R_d^*\big)^nu\|\geq C\|u\|,$$ for some constant $C>0$ and all $u\in {\mathcal H}.$
In particular,$$\|\big(\mu_1+\cdots+\mu_d-R_1^*-\cdots-R_d^*\big)^nu_k\|\geq C\|u_k\|=C.$$
If $k\to \infty$ we get $C=0$ which is a contradiction.
\begin{theorem}
Let  ${\bf \large R}=(R_1,\cdots,R_d)\in {\mathcal B}^{(d)}({\mathcal H})$ be an $(m,n)$-isosymmetric multioperators. then the following properties
hold.\par\vskip 0.2 cm
\noindent
$(1)$ $\sigma_{ja}({\bf\large R}) \subset \partial\mathbb{B}(\mathbb{C}^d) \bigcup \bigg\{  (\mu_1,\cdots,\mu_d) \in \mathbb{C}^d\;\;/ \;\displaystyle\sum_{1\leq k\leq d}\mu_k \in \mathbb{R}\;\bigg\}.$ \par \vskip 0.2 cm \noindent
$(2)$ $\sigma_{jp}({\bf\large R}) \subset \partial\mathbb{B}(\mathbb{C}^d) \bigcup \bigg\{  (\mu_1,\cdots,\mu_d) \in \mathbb{C}^d\;\;/ \;\displaystyle\sum_{1\leq k\leq d}\mu_k \in \mathbb{R}\;\bigg\}.$
\par \vskip 0.2 cm \noindent $(3)$  Let  $ \mu =\big(\mu_1,\cdots,\mu_d\big)$ and $
\mu^\prime=\big(\mu_1^\prime,\cdots,\mu_d^\prime \big) \in \sigma_{ja}({\bf\large R})$  such that
$\displaystyle\sum_{1\leq j\leq d}\mu_j\overline{\mu_j^\prime}\not=1$ and $\displaystyle\sum_{1\leq j\leq d}\big(\mu_j-\overline{\mu_j^\prime}\big)\not=0$. If
 $\{u_k\}_k$ and  $\{v_k\}_k$ are two sequences of unit vectors in ${\mathcal H}$ such that
$\|(R_j-\mu_j) u_k\|\longrightarrow 0 \;\hbox{and} \; \|(R_j-\mu_j^\prime)v_k\| \longrightarrow 0\;\;(\hbox{as}\;\;
k \longrightarrow \infty) \;\;\hbox{for}\;\;j=1,\cdots,d ,$ then
$$
\langle u_k |\; v_k\rangle \longrightarrow 0  \;\; (\hbox{as}\;\;
k\longrightarrow \infty).
$$
 \par \vskip 0.2 cm \noindent
 $(4)$  Let  $ \mu =(\mu_1,\cdots,\mu_d)$ and $
\mu^\prime=\big(\mu_1^\prime,\cdots,\mu_d^\prime\big) \in \sigma_{jp}({\bf\large R})$  such that
$$\displaystyle\sum_{1\leq j\leq d}\mu_j\mu_j^\prime\not=1\;\hbox{and}\; \;\displaystyle\sum_{1\leq j\leq d}\big(\mu_j-\overline{\mu_j^\prime}\big)\not=0.$$ If $\big(R_j-\lambda_j\big)u=0$ and $\big(R_j-\mu_j^\prime\big)v=0$ for $j=1,\cdots,d$, then
 $$\left\langle x\;|\;y\right\rangle=0.$$
\end{theorem}
\medskip
\end{proof}
\begin{proof}
(1)  Let $\mu= (\mu_1,\cdots,\mu_d)\in \sigma_{ja}({\bf \large R})$,
and  $\{u_k \}_{k}\subset\mathcal{H},$ such that
$\|u_k\|=1$ and  $\big(R_l-\mu_l\big ) u_k\longrightarrow 0$ for all
$l=1,\cdots,d$.  It is easy to see that
for all $\gamma_j\geq 0$ we have
$\big(R_j^{\gamma_j}-\mu_j^{\gamma_j})u_k\longrightarrow 0$ and
$\big({\bf\large R}^\gamma-\mu^\gamma\big)u_k\longrightarrow 0$  as
$k\longrightarrow \infty$. \par \vskip 0.2 cm \noindent
In the case that  ${\bf\large R}$ is an $(m,n)$-isosymmetric multioperators, then we have
\begin{eqnarray*}
0&=&\sum_{0\leq j\leq m}(-1)^{m-j}\binom{m}{j}\bigg(\sum_{|\gamma|=j}\frac{j!}{\gamma!}{\bf \large R}^{\ast \gamma}{\bf\large S}_n\big({\bf \large R}\big){\bf \large R}^\gamma\bigg)  u_k
\end{eqnarray*}
and so we may write \begin{eqnarray*}
0&=& \bigg\langle\sum_{0\leq j\leq m}(-1)^{m-j}\binom{m}{j}\bigg(\sum_{|\gamma|=j}\frac{j!}{\gamma!}{\bf \large R}^{\ast \gamma}{\bf\large S}_n\big({\bf \large R}\big){\bf \large R}^\gamma\bigg)  u_k\;|\;u_k\bigg\rangle\\&=&\sum_{0\leq j\leq m}(-1)^{m-j}\binom{m}{j}\sum_{|\gamma|=j}\frac{j!}{\gamma!}
\bigg\langle{\bf\large S}_n\big({\bf \large R}\big){\bf \large R}^\gamma  u_k\;|\;{\bf \large R}^{ \gamma}u_k\bigg\rangle
\\&=&\sum_{0\leq j\leq m}(-1)^{m-j}\binom{m}{j}\sum_{|\gamma|=j}\frac{j!}{\gamma!}
\bigg\langle{\bf\large S}_n\big({\bf \large R}\big)\bigg( \big({\bf \large R}^\gamma -\mu^\gamma\big) u_k+\mu^\gamma u_k\bigg)\;|\;\big({\bf \large R}^{ \gamma}-\mu^\gamma u_k\big)+\mu^\gamma u_k\bigg\rangle,
\end{eqnarray*}
which implies that
\begin{eqnarray*} &&\sum_{0\leq j\leq m}(-1)^{m-j}\binom{m}{j}\sum_{|\gamma|=j}\frac{j!}{\gamma!} |\mu|^{2\gamma}\lim_{k\to \infty}\bigg\langle  {\bf\large S}_n\big({\bf \large R}\big)u_k\;\big|\;u_k\bigg\rangle=0\\&\Rightarrow& \bigg(1-\sum_{1\leq j\leq d}|\mu_j|^2\bigg)^m\sum_{0\leq j\leq n}(-1)^j\binom{n}{j}
\lim_{k\to \infty }\bigg\langle \big(R_1+\cdots+R_d\big)^{n-j}u_k\big|\;\big(R_1+\cdots+R_d\big)^ju_k\bigg\rangle=0\\&\Rightarrow&
 \bigg(1-\sum_{1\leq j\leq d}|\mu_j|^2\bigg)^m\bigg( \mu_1+\cdots+\mu_d-\overline{\mu_1}-\cdots-\overline{\mu_d}\bigg)^n=0
 \\&\Rightarrow&
 \bigg(1-\sum_{1\leq j\leq d}|\mu_j|^2\bigg)^m\bigg( \mu_1+\cdots+\mu_d-\overline{\mu_1+\cdots+\mu_d}\bigg)^n=0
 \\&\Rightarrow&
 \bigg(1-\sum_{1\leq j\leq d}|\mu_j|^2\bigg)^m\bigg(2iIm\big( \mu_1+\cdots+\mu_d\big)\bigg)^n=0
 \\&\Rightarrow& \bigg(1-\sum_{1\leq j\leq d}|\mu_j|^2\bigg)^m=0\;\;\hbox{or}\;\bigg(2iIm\big( \mu_1+\cdots+\mu_d\big)\bigg)^n=0.
\end{eqnarray*}
This shows that   $\sigma_{ja}({\bf\large R}) \subset \partial\mathbb{B}(\mathbb{C}^d) \bigcup \bigg\{  (\mu_1,\cdots,\mu_d) \in \mathbb{C}^d\;\;/ \;\mu_1+\cdots+\mu_d\in \mathbb{R}\;\bigg\}.$\par \vskip 0.2 cm \noindent
$(2)$ Assume that  $\mu=(\mu_1,\cdots,\mu_d) \in \sigma_{jp}({\bf\large
R})$, then there exists  a nonzero vector $u$ for which  $\big(R_l - \mu_l\big)u = 0$ for  $l= 1,\cdots,d$. Since ${\bf\large R}$ is an $(m,n)$-isosymmetric multioperators, then it follows by a similar calculation
 as in the proof of statement $(1)$  that
$$\bigg(1-\sum_{1\leq j\leq d}|\mu_j|^2\bigg)^m\bigg(2iIm\big( \mu_1+\cdots+\mu_d\big)\bigg)^n=0,$$
and so that
$$\bigg(1-\sum_{1\leq j\leq d}|\mu_j|^2\bigg)^m=0\;\;\hbox{or}\;\bigg(2iIm\big( \mu_1+\cdots+\mu_d\big)\bigg)^n=0.$$
This justifies the statement $(2).$ \par \vskip 0.2 cm \noindent
$(3)$ Assume that $\{u_k\}_k$ and  $\{v_k\}_k$ be two sequences  in ${\mathcal H}$ such that $\|u_k\|=\|v_k\|=1$,
$$\|\big(R_l-\mu_l) u_k\|\longrightarrow 0 \;\hbox{and} \; \|\big(R_l-\mu_j^\prime\big)v_k\| \longrightarrow 0\;\;(\hbox{as}\;\;
k \longrightarrow \infty) \;\;\hbox{for }\;\;l=1,\cdots,d.$$ We have  $\displaystyle\lim_{k\rightarrow \infty}\big(R_l^{\gamma_l}-\mu_l^{\mu_l}\big)u_k=0$   and  $\displaystyle\lim_{k\rightarrow \infty}\big(R_l^{\gamma_l}-\mu_l^{\prime \gamma_l})v_k=0$,
which implies
$$\displaystyle\lim_{k\rightarrow \infty}\big({\bf\large R}^{\gamma}-\mu^{\gamma}\big)u_k=0\;\;\text{ and}\;\;\displaystyle\lim_{k\rightarrow \infty}({\bf\large R}^{\gamma}-\mu^{\prime \gamma})v_k=0.$$
Since ${\bf \large R}$ is an $(m,n)$-isosymmetric multioperators, it thus follows that
\begin{eqnarray*}
0&=&\lim_{k\rightarrow\infty}\bigg\langle \big(\sum_{0\leq j \leq m}(-1)^{m-j}\binom{m}{j}\sum_{|\gamma|=j}\frac{j!}{\gamma!}{\bf \large R}^{\ast \gamma}{\bf\large S}_n\big({\bf \large R}\big){\bf \large R}^\gamma \big)u_k \; \big|\;v_k\bigg\rangle\\&=&
\sum_{0\leq j \leq m}(-1)^{m-j}\binom{m}{j}\sum_{|\gamma|=j}\frac{k!}{\gamma!}\lim_{k\rightarrow\infty}\bigg\langle {\bf\large S}_n\big({\bf \large R}\big)\big({\bf \large R}^\gamma -{\mu}^\gamma+{\mu}^\gamma\big)u_k \; \big|\;\big({\bf \big(\large R}^{ \gamma}-\mu^{\prime \gamma}+\mu^{\prime \gamma}\big)v_k\bigg\rangle\\&=&
\sum_{0\leq j \leq m}(-1)^{m-j}\binom{m}{j}\sum_{|\gamma|=j}\frac{k!}{\gamma!}\mu^\gamma( \overline{\mu^\prime})^\gamma
\lim_{k\rightarrow\infty}\bigg\langle   {\bf\large S}_n\big({\bf \large R}\big)u_k \; |\;v_k\bigg\rangle\\&=&
\bigg(1-\mu_1.\overline{\mu_1^\prime}-\cdots-\mu_d.\overline{\mu_d^\prime}\bigg)^m\lim_{k\rightarrow\infty}\sum_{0\leq j\leq n}(-1)^j\binom{n}{j}\bigg\langle \big(R_1+\cdots+R_d\big)^{n-j}u_k\;\big|\; \big(R_1+\cdots+R_d\big)^j v_k\bigg\rangle
\\&=&\bigg(1-\mu_1.\overline{\mu_1^\prime}-\cdots-\mu_d.\overline{\mu_d^\prime}\bigg)^m\bigg(\mu_1+\cdots +\mu_d-\overline{\mu_1^\prime} -\cdots-\overline{\mu_d^\prime} \bigg)^n\lim_{k\rightarrow\infty}\big\langle u_k\;\big|\;v_k\big\rangle
\\&=& \bigg(1-\displaystyle\sum_{1\leq j\leq d}\mu_j\overline{\mu_j^\prime}\bigg)^m\bigg(\displaystyle\sum_{1\leq j\leq d}\big(\mu_j-\overline{\mu_j^\prime}\big)\bigg)^n\lim_{k\to \infty}\big\langle u_k\;\big|\;v_k\big\rangle,
\end{eqnarray*}
and this, since $1-\displaystyle\sum_{1\leq j\leq d}\mu_j\overline{\mu_j^\prime}\not=0\;\;\hbox{and}\;\;\displaystyle\sum_{1\leq j\leq d}\big(\mu_j-\overline{\mu_j^\prime}\big)\not=0$, implies $\displaystyle\lim_{k\to \infty}\big\langle u_k\;\big|\;v_k\big\rangle=0$ as claimed.
\noindent The proof of the statement $(4)$ follows from an argument similar to that used in $(3)$. This ends the proof.
\end{proof}
	\section{Acknowledgments}
	The authors extend their appreciation to the Deanship of Scientific Research at King Khalid University for funding this work through Small Research Project grant number (G.R.P.1/151/43).
\par \vskip 0.2 cm \noindent {\bf Data Availability}\par \vskip 0.2 cm \noindent
Data sharing is not applicable to this study as no data sets
were generated or analyzed during the current study. \par \vskip 0.2 cm \noindent
{\bf Conflicts of Interest}\par \vskip 0.2 cm \noindent
The authors declare that they have no conflicts of interest.




\end{document}